\documentclass[11pt,letterpaper]{amsart}
      \usepackage{amsmath}
      \usepackage{}
      \usepackage{amsfonts}
      \usepackage{mathrsfs}
      \usepackage{amsfonts, latexsym, rawfonts, amsmath, amssymb, amsthm}
      \usepackage[plainpages=false]{hyperref}

      \usepackage{fancyhdr}

      \usepackage{pdfsync}

      \usepackage{esint}
      \usepackage{graphics, graphicx}
      \usepackage{tikz}
      \usepackage{appendix}

      \numberwithin{equation}{section}
      \newcommand{\beq}{\begin{equation}}
      \newcommand{\eeq}{\end{equation}}
      \newcommand{\beqs}{\begin{eqnarray*}}
      \newcommand{\eeqs}{\end{eqnarray*}}
      \newcommand{\beqn}{\begin{eqnarray}}
      \newcommand{\eeqn}{\end{eqnarray}}
      \newcommand{\beqa}{\begin{array}}
      \newcommand{\eeqa}{\end{array}}

      \def\lra{\longrightarrow}

      \def\bc{\begin{center}}
      \def\ec{\end{center}}

      \def\begeq{\begin{equation}}
      \def\endeq{\end{equation}}
      \def\and{\quad{\rm and}\quad}

      \let\lra=\longrightarrow

      \def\mapright\#1{\, \smash{\mathop{\lra}\limits^{\#1}}\, }

      \newtheorem{prop}{Proposition}[section]
      \newtheorem{theo}[prop]{Theorem}
      \newtheorem{lem}[prop]{Lemma}
      
      \newtheorem{cor}[prop]{Corollary}

      \newtheorem{defi}[prop]{Definition}

\begin{document}

       \title{Steady  gradient Ricci solitons  with   nonnegative curvature operator  away from a compact set}

     \author{Ziyi  $\text{Zhao}^{\dag}$ and Xiaohua $\text{Zhu}^{\ddag}$}

 \address{BICMR and SMS, Peking
 University, Beijing 100871, China.}
 \email{ 1901110027@pku.edu.cn\\\ xhzhu@math.pku.edu.cn}

 \thanks {$\ddag$ partially supported  by National Key R\&D Program of China  2020YFA0712800,
 2023YFA1009900 and  NSFC 12271009.}
 \subjclass[2000]{Primary: 53E20; Secondary: 53C20,  53C25, 58J05}

 \keywords{Steady  gradient Ricci soliton, Ricci flow,  ancient  $\kappa$-solution,   Bryant  Ricci soliton}

      \begin{abstract} Let $(M^n,g)$  $(n\ge 4)$ be a  complete noncompact $\kappa$-noncollapsed steady Ricci soliton with $\rm{Rm}\geq 0$ and $\rm{Ric}> 0$ away from a compact  set  $K$ of $M$.  We prove that  there is no any
       $(n-1)$-dimensional compact  split limit Ricci flow of type I  arising from  the blow-down  of $(M, g)$,   if  there is an $(n-1)$-dimensional  noncompact  split limit Ricci flow.  Consequently,   the compact  split limit ancient flows of type I and type II cannot occur simultaneously  from the  blow-down.  As an application,  we prove that $(M^n,g)$   with $\rm{Rm}\geq 0$  must be isometric the  Bryant Ricci soliton up to scaling,  if  there exists a sequence of   rescaled Ricci flows $(M,g_{p_i}(t); p_i)$ of $(M,g)$ converges subsequently to a family of shrinking quotient cylinders.

   \end{abstract}

       \date{}

    \maketitle

   % \tableofcontents

  %  \setcounter{section}{-1}

    \section{Introduction}

   Let $(M^n, g; f )$  $(n\ge 4)$ be a complete  noncompact $\kappa$-noncollapsed  steady  gradient Ricci soliton with curvature operator
     %${\rm Ric}>0$ and
       $\rm{Rm}\geq 0$ away from a compact  set $K$ of $M$.   Let $g(\cdot, t)=\phi^*_{t}(g)$ $(t\in (-\infty, \infty))$  be an induced ancient Ricci flow of  $(M,g)$,
       where $\phi_{t}$ is a family of transformations generated by the gradient vector field $-\nabla f$.
         For any  sequence of $p_i\in M$ $(\to \infty)$,  we consider the rescaled Ricci  flows $(M, g_{p_i}(t); p_i)$,  where
        \begin{align}\label{rescaling-flow}g_{p_i}(t)=r_i^{-1} g( \cdot, r_it),
  \end{align}
  $r_iR(p_i)=1$.
  By a version of  Perelman's compactness theorem for  ancient  $\kappa$-solutions \cite[Proposition 1.3]{ZZ-4d},
  we know that $(M, g_{p_i}(t); p_i)$   converge  subsequently to a  splitting flow $(N \times \mathbb{R}, \bar {g}(t); p_\infty)$ in the Cheeger-Gromov sense, where
  \begin{align}\label{splitt-solution} \bar g(t)= h(t) +ds^2, ~ {\rm on}~ N\times \mathbb R,
    \end{align}
 and $h(t)$  ($t\in (-\infty, 0]$)  is an ancient   $\kappa$-solution on an  $(n-1)$-dimensional $N$.
 For simplicity, we call $(N,  h(t))$ a compact   split limit  flow (by the  blow-down) if $N$ is compact, otherwise, $(N,  h(t))$ is a  noncompact   split limit  flow   if $N$ is noncompact.

  In this paper, we extend  our previous result  \cite[Corollary 0.3]{ZZ-4d} from the dimension $4$ to any dimension. Namely, we prove

   \begin{theo}\label{Zhao-Z-high-dim}
Let $(M^n, g)$ $(n\ge 4)$  be a  noncompact $\kappa$-noncollapsed steady gradient Ricci soliton with nonnegative curvature operator. Suppose that there exists a sequence  of  $p_i\in M$ $( \to \infty)$ such that the  rescaled Ricci flows $(M,g_{p_i}(t); p_i)$ of $(M,g)$ converge subsequently to a family of shrinking quotient cylinders.  Then $(M, g)$ is isometric to the $n$-dimensional Bryant Ricci soliton up to scaling.
\end{theo}

Theorem \ref{Zhao-Z-high-dim} is also an improvement  of Brendle \cite[Theorem 1.2]{Bre-high},  and Deng-Zhu  \cite[Theorem 1.3] {DZ-epsilon} and \cite[Lemma 6.5]{DZ-SCM}.  In  the above  papers,   one  shall assume that for  any sequence  of  $p_i\in M \rightarrow \infty$  the  rescaled flows $(M,g_{p_i}(t); p_i)$ of $(M,g)$ converge subsequently to a family of shrinking  cylinders. We note that the
nonnegativity condition of curvature operator can be weakened as the   nonnegativity  of sectional curvature  when $n=4$ \cite[Corollary 0.3]{ZZ-4d}.

Our  proof of Theorem \ref{Zhao-Z-high-dim} depends on the following classification of split ancient $\kappa$-solutions.

\begin{theo}\label{noncompact-ancient-solution}
Let $(M^n,g)$  $(n\ge 4)$ be a noncompact $\kappa$-noncollapsed steady Ricci soliton with $\rm{Rm}\geq 0$ and $\rm{Ric}> 0$ on  $M\setminus K$.  Suppose that there exists a sequence of rescaled Ricci flows $(M,g_{p_i}(t); p_i)$, which converges subsequently to a  splitting Ricci flow $(N\times \mathbb R, \bar h(t); p_\infty)$ as in (\ref{splitt-solution}) for some noncompact ancient $\kappa$-solution $(N, h(t))$.  Then there is no any compact  split limit Ricci flow of type I    arising from  the blow-down  of $(M, g)$.
\end{theo}

 Recall that
  a compact  ancient solution $(N, h(t))$ $(t\in (-\infty, 0])$  of type I means that it satisfies
         \begin{align*}
             \sup_{N\times (-\infty,0]} (-t)|R(x,t)|<\infty.
         \end{align*}
       Otherwise,   it is called type \uppercase\expandafter{\romannumeral2},  i.e.,   it satisfies
         \begin{align*}
             \sup_{M\times (-\infty,0]} (-t)|R(x,t)|=\infty.
         \end{align*}

By Theorem \ref{noncompact-ancient-solution}, we also prove

\begin{cor}\label{notype1-2}
Let $(M^n,g)$ $(n\ge 4)$  be a noncompact $\kappa$-noncollapsed steady Ricci soliton as  in Theorem \ref{noncompact-ancient-solution}. Then the compact  split limit ancient flows of type I and type II cannot occur simultaneously  from the  blow-down  of $(M, g)$.
\end{cor}

 By Theorem \ref{noncompact-ancient-solution} and Corollary \ref{notype1-2}, we conclude   that for any  noncompact $\kappa$-noncollapsed steady Ricci soliton with $\rm{Rm}\geq 0$ and $\rm{Ric}> 0$ on $M\setminus K$,   either all split limit flows $(N, h(t))$ as in (\ref{splitt-solution})  are compact and of type I,  or there exists  at least one  noncompact split limit flow.  The conclusion  can be regarded as a generalization of Chow-Deng-Ma's result \cite[Theorem 1.3, Claim 6.4]{CDM} from  dimension $4$ to any dimension.

Compared to the proof of  \cite[Theorem  0.2]{ZZ-4d} for the $4d$  steady Ricci soliton, we shall modify the argument  for  one of Theorem \ref{noncompact-ancient-solution} in the case of  higher dimensions,    since we have no classification result for codimemsional one  compact ancient $\kappa$-solutions of  type II as  dimension $3$ \cite{Bre-3d-noncpt, ABDS, BK}. We will get a distance estimate between two compact level sets, each of which has  a large diameter,  to see  (\ref{sequence-bound}). All main results  will be proved in Section 3.

    \section{Preliminaries}

    In this section, we  review  some results proved  in our previous  article  \cite{ZZ-4d}, which will be also used in this paper.  As in  \cite{ZZ-4d}, we will always assume that
    $(M^n, g; f )$  $(n\ge 4)$ is  a noncompact $\kappa$-noncollapsed  steady  gradient Ricci soliton with curvature operator
    $\rm{Rm}\geq 0$  on $M\setminus K$.

    The following result can be regarded as a Harnack  type estimate for  steady Ricci solitons.

    \begin{lem}(\cite[Lemma 1.3]{ZZ-4d}) \label{compact-curvature-bound}
       Let $(M^n,g)$ be a complete noncompact $\kappa$-noncollapsed steady Ricci soliton  with $\rm{Rm}\geq 0$ on $M\setminus K$.  Let $\{p_i\} \to \infty$ be a sequence in $(M, g)$. Then, for any $q_i \in B_{g_{p_i}}(p_i, D)$, there exists $C_0(D) > 0$ such that
       \begin{align}\label{curvature-control}
           C_0^{-1}R(p_i) \leq R(q_i) \leq C_0 R(p_i).
       \end{align}
   \end{lem}

    \subsection{A decay estimate of curvature}

    By   \cite[Proposition 1.3]{ZZ-4d}, for any sequence  $\{p_i\} \to \infty$,
  rescaled  Ricci flows  $(M, g_{p_i}(t); p_i)$   converge subsequently to a  splitting  Ricci flow $(N \times \mathbb{R}, \bar {g}(t); p_\infty)$ in the Cheeger-Gromov sense, where
$ \bar {g}(t)= h(t)+dr^2$  as in
 (\ref{splitt-solution}).  We assume that the ancient   $\kappa$-solution $(N, h(t))$ is compact. Namely,  there is a constant $C$ such that
    \begin{align}\label{bound-h}
       \text{Diam}(N, h(0)) \leq C.
   \end{align}
   Then we have the following curvature decay estimate.

   \begin{lem}(\cite[Lemma 2.2]{ZZ-4d})\label{compact-R-decay}
       Let $(M^n, g)$ be a complete noncompact  steady Ricci soliton with $\text{Ric} > 0$ and  $\rm{Rm}\geq 0$ on $M\setminus K$.  Suppose that there exists a sequence of $p_i \rightarrow \infty$ such that the split $(n-1)$-dimensional ancient $\kappa$-solution $(N, {h}(t))$  satisfies (\ref{bound-h}). Then the scalar curvature  of $(M, g)$ decays to zero uniformly. Namely,
       \begin{align}\label{R-decay}
           \lim_{x \rightarrow \infty}R(x) = 0.
       \end{align}
   \end{lem}

   By  Lemma \ref{compact-R-decay} and  the normalization  identity
    \begin{align}\label{scalar-equ}
        R + |\nabla f|^2 = 1,
    \end{align}
  we have
    \begin{align}\label{gradient-esti}
     |\nabla f(x)|\to 1~{\rm as}~ \rho(x)\to\infty.
     \end{align}
      Moreover, by \cite[Lemma 2.2]{DZ-SCM} (or \cite[Theorem 2.1]{CDM}), $f$ satisfies
       \begin{align}\label{linear-f}
           c_1\rho(x)\le f(x)\le c_2\rho(x)
       \end{align}
       for two constants $c_1$ and $c_2$. Hence, the integral curve $\gamma(s)$ generated by $X=\nabla f$ extends to the infinity as $s\rightarrow \infty$.

 Based on  Lemma \ref{compact-curvature-bound} and
  Lemma \ref{compact-R-decay},   by studying  the level set geometry of $(M, g)$,  we prove

   \begin{prop}(\cite[Proposition  2.7]{ZZ-4d})\label{compact-levelset-compact}
       Let $(M, g)$ be the steady Ricci soliton as  in Lemma \ref{compact-R-decay} and $(N\times \mathbb{R},h(t)+ds^2;p_\infty)$ the splitting limit flow of $(M,g_{p_i}(t);p_i)$, which satisfies (\ref{bound-h}). Then there exists $C_0(C)>0$ such that for any  sequence of $q_i\in  f^{-1} (f(p_i))$ the splitting limit flow  $( h'(t)+ds^2, N'\times \mathbb{R};  q_\infty)$ of rescaled flows $(M,g_{q_i}(t); q_i)$ satisfies
       \begin{align}\label{diam-estimate}
           {\rm Diam}(h'(0))\leq C_0.
       \end{align}
   \end{prop}

   \subsection{A classification of split compact ancient solutions}

In  case that   all  split   ancient $\kappa$-solution $(N, h(t))$  satisfies $(\ref{bound-h})$,  we can  classify  $(N, h(t))$.

  \begin{prop}(\cite[Proposition  4.1]{ZZ-4d})\label{compact-limit-ancient-solution}
      Let $(M, g)$ be a  steady Ricci soliton as  in Lemma \ref{compact-R-decay}.   Suppose that  there is a uniform constant $C$  such that   all  split   ancient $\kappa$-solution $(N, h(t))$  satisfies $(\ref{bound-h})$.
          Then  every $h(t)$ must be an  ancient $\kappa$-solution of type \uppercase\expandafter{\romannumeral1}.
 \end{prop}

 Compact ancient $\kappa$-solution of  type \uppercase\expandafter{\romannumeral1} has  been classified as follows (cf. \cite[Theorem 7.34]{CLN}, \cite{BW}, \cite{Ni}).

  \begin{lem}\label{classification-type1}
      Suppose that $(N^{n-1}, h(t))$ of $(n-1)$-dimension is a compact  ancient $\kappa$-solution of type \uppercase\expandafter{\romannumeral1} with ${\rm R_m}\ge0$. Then
      \begin{align}\label{N-split}
          &( N^{n-1}, h(t))\notag\\
          &\cong (N_1,h_1(t))\times\cdots\times (N_k,h_k(t))\times (\hat N_1, \hat h_1(t))\times \cdots \times (\hat N_\ell, \hat h_\ell(t)),
      \end{align}
      where each $(N_i,h_i(t))$ is a family of shrinking quotients of a closed symmetric space with nonnegative curvature operator, and each $(\hat N_j, \hat h_j(t))$ is a family of shrinking round quotient spheres.
  \end{lem}

To prove Proposition \ref{compact-limit-ancient-solution},  we shall exclude the existence of  ancient $\kappa$-solutions of type \uppercase\expandafter{\romannumeral2}.  Actually, we prove the following  diameter estimate for such  ancient  solutions.

  \begin{lem}(\cite[Lemma 4.3]{ZZ-4d})\label{type2-diameter}
      Let $(N^{n-1},h(t))$ be an $(n-1)$-dimensional compact ancient $\kappa$-solution of type \uppercase\expandafter{\romannumeral2}.
      % to the Ricci flow with positive curvature operator.
      Then for any sequence $t_k\rightarrow -\infty$, it holds
      \begin{align*}
          R_{min}(t_k){\rm Diam}(h(t_k))^2\rightarrow \infty,
      \end{align*}
      where $R_{min}(t)=\min\{R(h( \cdot, t)\}$.
      Consequently,
      \begin{align}\label{infty-diam}
          \lim_{t\rightarrow -\infty} R_{min}(t){\rm Diam}(h(t))^2\rightarrow \infty.
      \end{align}
  \end{lem}

Lemma \ref{type2-diameter}
can be regarded as a higher dimensional version of \cite[Lemma 2.2]{ABDS}.

  \section{Proofs of  main results}

  In this section, we prove Theorem \ref{noncompact-ancient-solution} as well as and Corollary \ref{notype1-2} and Theorem \ref{Zhao-Z-high-dim}.  First, we recall the following definition  introduced by Perelman (cf. \cite{P2}).

  \begin{defi}\label{epsilonclose}
      For any $\epsilon>0$, we say a pointed Ricci flow $\left(M_1, g_1(t); p_1\right), t \in$ $[-T, 0]$, is $\epsilon$-close to another pointed Ricci flow $\left(M_2, g_2(t); p_2\right), t \in[-T, 0]$, if there is a diffeomorphism onto its image $\bar{\phi}: B_{g_2(0)}\left(p_2, \epsilon^{-1}\right) \rightarrow M_1$, such that $\bar{\phi}\left(p_2\right)=p_1$ and $\left\|\bar{\phi}^* g_1(t)-g_2(t)\right\|_{C^{\left[\epsilon^{-1}\right]}}<\epsilon$ for all $t \in\left[-\min \left\{T, \epsilon^{-1}\right\}, 0\right]$, where the norms and derivatives are taken with respect to $g_2(0)$.
  \end{defi}

   By the compactness of  rescaled  Ricci flows  \cite[Proposition 1.3]{ZZ-4d},  we know that for any $\epsilon>0$, there exists a compact set  $D(\epsilon)>0$, such that for any $p\in M\setminus D$, $(M,g_p(t);p)$ is $\epsilon$-close to a splitting  flow $(h_p(t)+ds^2; p)$,  where  $h_p(t)$  is an $(n-1)$-dimensional ancient $\kappa$-solution.
Since  the  $\epsilon$-close  splitting  flow $(h_p(t)+ds^2; p)$ may not be unique for a point  $p$,
   we  may introduce a function on $M$ for each  $\epsilon$  as in \cite{Yi-flyingwings},
 \begin{align}\label{notation-f}F_{\epsilon}(p)=\inf_{h_p} \{ {\rm{Diam}} (h_p(0))\in(0,\infty)  \}.
 \end{align}
 For simplicity, we always  omit the subscribe $\epsilon$ in the function  $F_{\epsilon}(p)$ below.

 %We will first prove Theorem \ref{noncompact-ancient-solution}.
\subsection{Proof of Theorem \ref{noncompact-ancient-solution}}

We  use the argument by contradiction. On the contrary,   we suppose that there exists a  sequence of rescaled Ricci flows $(M,g_{q_i}(t);q_i)$ ( $q_i\to\infty$), which converges to a limit Ricci flow   $(N'\times \mathbb{R},h'(t)+ds^2; q_\infty)$, where $(N',h'(t))$ is a compact ancient   $\kappa$-solution of type I.     Then  by Lemma  \ref{classification-type1},  there exists a constant $C_0 > 0$ such that for any small $\epsilon > 0$ it holds,
\begin{align*}
    F(q_i) =F_\epsilon(q_i)\leq C_0.
\end{align*}
By Proposition \ref{compact-levelset-compact},  it follows
 \begin{align}\label{type1-bound}
    F(q) \leq C,
\end{align}
for all $q \in f^{-1}(f(q_i))$. We note that the constant $C$ is uniform by the classification result, Lemma \ref{classification-type1}, i.e., it is  independent of the sequence of rescaled Ricci flows  with a limit Ricci flow,  which is a  compact    ancient   $\kappa$-solution of type I.

 On the other hand, for the sequence of $(M,g_{p_i}(t); p_i)$ in Theorem \ref{noncompact-ancient-solution}, we can choose a point $p_{i_0} \in \{p_i\}$ such that
\begin{align}\label{F_0}
    F(p_{i}) > 100C,
\end{align}
for all $i>i_0$. Let   $ \hat X=\frac{\nabla f}{|\nabla f|}$ and $\Gamma(s)$ be an  integral curve of $ \hat X$  starting from  $p_{i_0}$,  i.e.,  $\Gamma(0) = p_{i_0}$.  We note that $\Gamma(s)$   tends to the  infinity  by  \eqref{gradient-esti} and \eqref{linear-f},   since Lemma \ref{compact-R-decay} holds.   Thus by $\eqref{type1-bound}$ and $\eqref{F_0}$,  we can choose two sequences  $\{p^1_i\}$ and $\{p^2_i\}$  of points in  $\Gamma(s)$ to the infinity, which satisfy the following properties:
\begin{align}\label{sequence-p}
&1) f(p^1_i)<f(p^2_i);\notag\\
&2) f(p^1_i)< f(p^1_{i+1}),   f(p^2_i)< f(p^2_{i+1});\notag\\
&3) f(p^1_i)<  f(q_i)< f(p^2_{i});\notag\\
&4)  F(p^1_{i}) = F(p^2_{i}) = 10C\notag\\
&5) F(p) \le 10C, ~ \forall ~p\in \Gamma(s)~ {\rm with}~ f(p^1_i)\le f(p)\le f(p^2_{i}).
\end{align}
Thus there are $s^1_i, s^2_i, s_i$ with $s^1_i< s_i< s^2_i$ such that
\begin{align}\label{si-points}\Gamma(s^1_i) = p^1_i, \Gamma(s^2_i) = p^2_i~{\rm and}~\Gamma(s_i) = q'_i,
\end{align}
where $q'_i=\Gamma(s)\cap f^{-1}(f(q_i))$,  see  Figure \ref{fig1}.

\begin{figure}[h]
    \centering
    \includegraphics[width=1.0\textwidth]{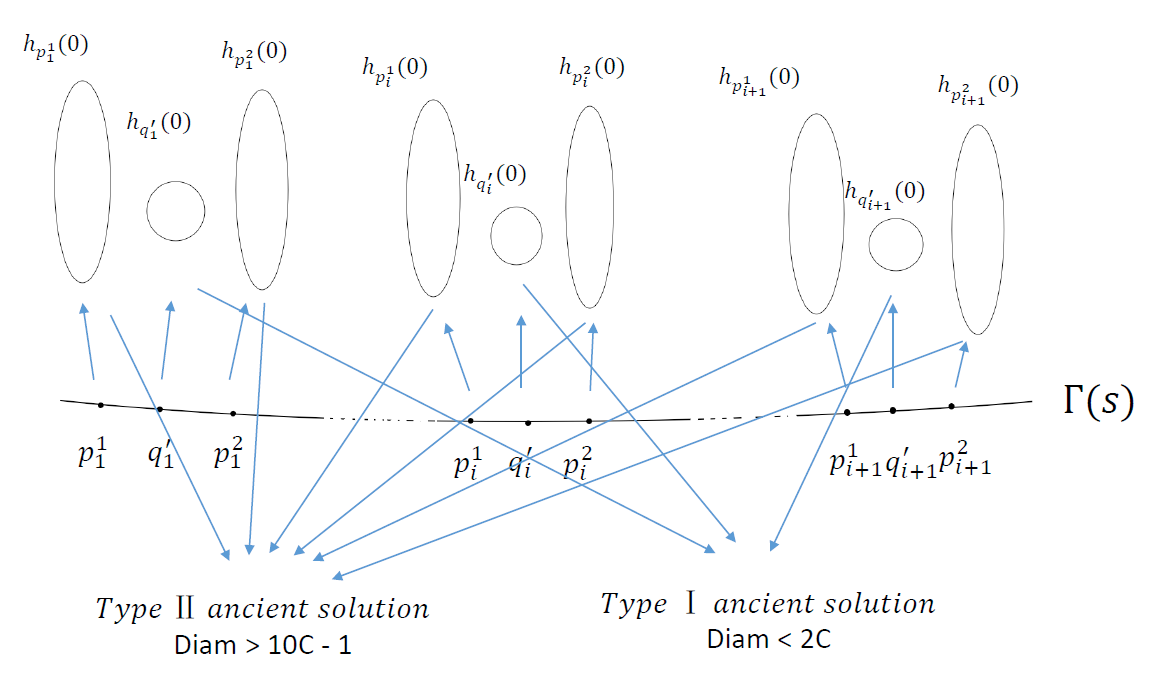}
    \caption{}
    \label{fig1}
\end{figure}

 By \cite[Proposition 1.3]{ZZ-4d} and the relation 4) in (\ref{sequence-p}),   $(M, g_{p^1_i}(t) ; p^1_i)$ converges  subsequently to $(N_1\times \mathbb{R},h_1(t)+ds^2; p^1_\infty)$, where $h_1(t)$ satisfies
\begin{align*}
    {\rm Diam}(h_1(0)) \in [10C-1, 10C+1].
\end{align*}
 Thus, by $\eqref{type1-bound}$, we see that $h_1(t)$ must be a compact  ancient $\kappa$-solution of type II.  By Lemma \ref{type2-diameter}, it follows
\begin{align}\label{type2-infty}
    R(p^1_\infty, t_k){\rm Diam}(h_1(t_k))^2 \rightarrow \infty
\end{align}
for any $t_k \to -\infty$. Hence, there exists a $k_0$ such that
\begin{align}\label{type2-k0}
    R(p^1_\infty, t_{k_0})^{1/2}{ \rm Diam}(h_1(t_{k_0})) > 100C.
\end{align}
We fix $t_{k_0}$ so that  $\epsilon^{-1}>-10 t_{k_0}$ as long as $\epsilon<< 1$.

 Next we claim there exists a constant $A > 0$ such that
\begin{align}\label{sequence-bound}
    s^2_i - s^1_i < A R(p^1_i)^{-1}
\end{align}
for all $i \gg 1$.

 Suppose that the above  claim is not true. Then by taking a subsequence, we may assume
\begin{align}\label{contra-s}
   R(p^1_i)(s^2_i - s^1_i) \to \infty.
\end{align}
%Since $h'_1(t)$ is a compact ancient $\kappa$-solution of type II, then by

   Recall that  $\left\{\phi_t\right\}_{t \in(-\infty, \infty)}$ is the flow of $-\nabla f$ with $\phi_0$ the identity and $(g(t), \Gamma(s))$ is isometric to $\left(g, \phi_t(\Gamma(s))\right)$.  Then
    \begin{align*}
        \phi_t(\Gamma(s))=\Gamma\left(s-\int_0^t|\nabla f|\left(\phi_\mu(\Gamma(s))\right) d \mu\right)
    \end{align*}
    Let $T_{k_0,i}=t_{k_0} R(p_i^1)^{-1}=t_{k_0} R\left(\Gamma\left(s^1_i\right)\right)^{-1}$ and
    $$ s'_i=s^1_i -\int_0^{T_{k_0,i}}|\nabla f|\left(\phi_\mu\left(\Gamma\left(s^1_i\right)\right)\right) d \mu.$$
    It follows
    \begin{align}\label{s-gap}
        s'_i - s^1_i \leq -T_{k_0,i} = -t_{k_0} R\left(\Gamma\left(s^1_i\right)\right)^{-1}.
    \end{align}
    By combining \eqref{contra-s} and \eqref{s-gap}, we see  that $s'_i \in [s^1_i,s^2_i]$ for $i \gg 1$.
    Hence, by  the relation 5) in (\ref{sequence-p}), we obtain
\begin{align}\label{F-s-2}
    F(\Gamma_1(s'_i)) \leq 10C.
\end{align}

    By the isometry, we have
\begin{align}\label{levelset-isometric}
        (M,R(\Gamma(s'_i))g;\Gamma(s'_i))&\cong (M,R(\Gamma(s^1_i),T_{k_0,i})g(T_{k_0,i});\Gamma(s^1_i))\notag\\
        &\cong (M,\frac{R(\Gamma(s^1_i),T_{k_0,i})}{R(\Gamma(s^1_i))}R(\Gamma(s^1_i))g(T_{k_0,i});\Gamma(s^1_i)).
    \end{align}
On the other hand,
  from the proof of  \cite[Proposition 3.6]{ZZ-4d}, we know that for each $ \Gamma(s^1_i)$, there exists a $(n-1)$-dimensional compact
  ancient $\kappa$-solution $h_{\Gamma(s^1_i)}(t)$ such that
 \begin{align}\label{s0-close}
     &  (M,R(\Gamma(s^1_i))g(R(\Gamma(s^1_i))^{-1}t);\Gamma(s^1_i))\notag\\
     & \overset{\epsilon-\text{close}}{\sim} (N\times \mathbb{R},h_{\Gamma(s^1_i)}(t)+ds^2; \Gamma(s^1_i)).
   \end{align}
   Since
$R(\Gamma_1(s^1_i),T_{k_0,i})\leq R(\Gamma(s^1_i))$  by the monotonicity of scalar curvature along $\Gamma(s)$, by (\ref{levelset-isometric}),
we get
\begin{align}\label{s1-close1}
    &(M,R(\Gamma(s'_i))g;\Gamma(s'_i))\notag\\
    & \overset{\epsilon-\text{close}}{\sim} (N\times \mathbb{R},\frac{R(\Gamma(s^1_i),T_{k_0,i})}{R(\Gamma(s^1_i))}h_{\Gamma(s^1_i)}(t_{k_0})+ds^2; \Gamma(s^1_i)).
\end{align}
Note that  there are also another  $(n-1)$-dimensional compact
ancient $\kappa$-solutions $h_{\Gamma(s'_i)}(t)$ corresponding to   the point $\Gamma(s'_i)$  as in (\ref{s0-close}) such that
\begin{align}\label{s1-close2}
    (M,R(\Gamma(s'_i))g;\Gamma(s'_i)) \overset{\epsilon-\text{close}}{\sim} (h_{\Gamma(s'_i)}(0)+ds^2, \Gamma(s'_i)).
\end{align}
Hence, combining $(\ref{s1-close1} )$ and $(\ref{s1-close2})$, we derive
\begin{align}\label{h-sequence-limit}
    h_{\Gamma(s'_i)}(0)& \overset{\epsilon-\text{close}}{\sim} \frac{R(\Gamma(s^1_i),T_{k_0,i})}{R(\Gamma(s^1_i))}h_{\Gamma(s^1_i)}(t_{k_0}) \notag\\
    &  \overset{\epsilon-\text{close}}{\sim} R_h(\Gamma(s^1_i),t_{k_0})h_{\Gamma(s^1_i)}(t_{k_0}).
\end{align}

By the convergence, we have
\begin{align}\label{two-ancients}
    h_{\Gamma(s^1_i)}(t_{k_0}) \overset{\epsilon-\text{close}}{\sim} h_1(t_{k_0}), ~{\rm as}~i\gg1.
\end{align}
 It follows
\begin{align}\label{Rh-sequence-limit}
    R_h(\Gamma(s^1_i),t_{k_0}) \overset{\epsilon'(\epsilon)-\text{close}}{\sim} R(p^1_\infty, t_{k_0}),
\end{align}
 where $R$ is the scalar curvature w.r.t. $h_1$.
  Thus by  \eqref{h-sequence-limit} and (\ref{type2-k0}), we estimate
\begin{align}\label{F-s-1}
    F(\Gamma(s'_i)) &\geq {\rm Diam} (h_{\Gamma(s'_i)}(0)) - \epsilon \notag\\
    &\geq {\rm Diam}(R_h(\Gamma(s^1_i),t_{k_0})h_{\Gamma(s^1_i)}(t_{k_0})) - \epsilon(1+\epsilon') \notag\\
    &> 98C.
\end{align}
But this is impossible by (\ref{F-s-2}).
Hence,   (\ref{sequence-bound}) must be true.

Now we can finish the proof of  Theorem \ref{noncompact-ancient-solution}.   By   (\ref{sequence-bound}) and the relation 3) in  (\ref{sequence-p}), we have
\begin{align}\label{s-type1}
    s^2_i - s_i \leq s^2_i - s^1_i \leq A R\left(\Gamma\left(s^1_i\right)\right)^{-1} \leq A R\left(\Gamma\left(s_i\right)\right)^{-1},
\end{align}
where the last inequality follows from the monotonicity of $R$ along $\Gamma(s)$. Since $(M, g_{\Gamma(s_i)}(t); \Gamma(s_i))$ converges subsequently to the  limit flow $(N'\times \mathbb{R}, h'(t)+ds^2; q_\infty)$,  by the shrinking property  of  type I solution  $(N', h'(t))$   in Lemma \ref{classification-type1}, we know
\begin{align}\label{type1-scaled-diameter}
    {\rm{Diam}}(R_{h'}(q_\infty,t)h'(t)))\leq 2C, ~\forall~t \leq 0.
\end{align}
 Then as  in $(\ref{h-sequence-limit})$, we get
\begin{align}\label{type1-ancient}
    h_{\Gamma(s^2_i)}(0) \overset{\epsilon-\text{close}}{\sim}  R_{h}(\Gamma(s_i),t'_i)h_{\Gamma(s_i)}(t'_i)
\end{align}
for some $t'_i$,  where  $-t'_i < 2A$   by \eqref{s-type1}.   Thus ,  combining $(\ref{type1-scaled-diameter})$ and (\ref{type1-ancient}) and by the convergence of  $(M, g_{\Gamma(s_i)}(t); \Gamma(s_i))$, we obtain
\begin{align}\label{type1-F-1}
    F(\Gamma(s^2_i)) \leq 3C.
\end{align}
But this is impossible  since by the relation 4) in (\ref{sequence-p}) it holds
\begin{align}\label{type1-F-2}
    F(\Gamma(s^2_i)) =   F(p^2_i) =10C.
\end{align}
Therefore, we prove the  theorem.

%\end{proof}

\subsection{Proofs of Corollary \ref{notype1-2} and Theorem \ref{Zhao-Z-high-dim}}

\begin{proof}[Proof of Corollary \ref{notype1-2}]
  Suppose that  there exist two splitting limit flows $(N_1 \times \mathbb{R}, h_1(t) + ds^2; p^1_\infty)$ and $(N_2 \times \mathbb{R}, h_2(t) + ds^2; p^2_\infty)$ of rescaled flows $(M, g_{p^1_i}(t); p^1_i)$ and $(M, g_{p^2_i}(t); p^2_i),$ respectively, such that $(N_1, h_1(t))$ is a compact ancient $\kappa$-solution of type \uppercase\expandafter{\romannumeral1}, and $(N_2, h_2(t))$ is another compact ancient $\kappa$-solution of type \uppercase\expandafter{\romannumeral2}. We claim that
  \begin{equation}\label{infty-diameter}
    \limsup_{\epsilon \to 0} \limsup_{p \to \infty} F_\epsilon(p) = \infty.
  \end{equation}

  On the contrary,   there will be  a uniform constant $C$ such that all split ancient $\kappa$-solutions $(N^{n-1}, h(t))$ of $(n-1)$-dimension satisfy $(\ref{bound-h})$. Then by  Proposition \ref{compact-limit-ancient-solution}, it follows that every split limit flow $h(t)$ must be an ancient $\kappa$-solution of type \uppercase\expandafter{\romannumeral1}.  But this is impossible since  $h_2(t)$ is a compact ancient $\kappa$-solution of type \uppercase\expandafter{\romannumeral2}.  Thus the claim  is true.

  By  \eqref{infty-diameter},  it is easy to see that  there is a sequence of pointed flows $(M, g_{q_i}(t); q_i)$, which converges subsequently to a splitting Ricci flow $(N' \times \mathbb{R}, h'(t) + ds^2; q_\infty)$ for some noncompact ancient $\kappa$-solution $(N', h'(t))$.  Thus by Theorem \ref{noncompact-ancient-solution},   there cannot exist a compact splitting limit flow of type \uppercase\expandafter{\romannumeral1}.    But this is impossible since   $h_1(t)$ is  a compact  ancient $\kappa$-solution of type  \uppercase\expandafter{\romannumeral1}.  This proves  the corollary.
\end{proof}

\begin{proof}[Proof of Theorem \ref{Zhao-Z-high-dim}] By the assumption,  the  $(n-1)$-dimensional split ancient flow
 $(N,h(t))$ of limit of $(M,g_{p_i}(t),p_i)$ is  a family of shrinking  round quotient spheres. Namely,  $(N,h(0))$ is a quotient  of
round  sphere, so it  is of type \uppercase\expandafter{\romannumeral1}.  We first show that  $(M,g)$ has positive Ricci curvature on $M$.

 On the contrary,   ${\rm Ric}(g)$   is not strictly positive.   We note that   the scalar curvature  $R(p_i)$ decaying to zero is still true in the proof of  Lemma \ref{compact-R-decay} without  ${\rm Ric}(g)>0$ away from a compact set of $M$.
  Then as in the proof of \cite[Lemma 4.6]{DZ-JEMS}, we see that  $X_i=R(p_i)^{-\frac{1}{2}}\nabla f\to X_\infty$ w.r.t.  $(M,g_{p_i}(t); p_i)$,  where $X_\infty$ is a non-trivial parallel vector field.  Thus  according to the argument  in  the proof of \cite [Theorem 1.3]{DZ-JEMS},    the universal cover of $(N,h(t)) $ must split off a flat factor $\mathbb{R}^d$ ($d\geq 1$).  However, the universal cover of $N$ is $S^{n-1}$.  This is a contradiction!  Hence, we conclude that ${\rm Ric}(g)>0$ on $M$.

  Now we divide   into two cases to prove the theorem.

  Case 1:
  $$\limsup_{p\rightarrow \infty}F_\epsilon(p)<C. $$
  for any $\epsilon<1$.  Then by  Proposition \ref{compact-limit-ancient-solution} and Lemma \ref {classification-type1},   all  $(n-1)$-dimensional split ancient $\kappa$-solutions   $(N',h'(t))$ are of
      type \uppercase\expandafter{\romannumeral1},   and each  of them  is one described  in    (\ref{N-split}).

  Case 2:
  $$\limsup_{\epsilon\to 0}\limsup_{p\rightarrow \infty}F_\epsilon(p)=\infty.$$
  In this case,  there will exist  a sequence of pointed flows $(M,g_{q_i}(t);q_i)$, which  converges  subsequently to a splitting Ricci flow $(N'\times\mathbb{R}, h'(t)+ds^2; q_\infty)$ for some noncompact ancient $\kappa$-solution $(N', h'(t))$.  But this is impossible by Theorem \ref{noncompact-ancient-solution}  since we  already have had  a  split ancient limit  flow $(N,h(t))$ of type  \uppercase\expandafter{\romannumeral1}.  Thus Case 2 can be excluded.

It remains to show  that every split limit flow  $(N',h'(t))$ in Case 1 is  in fact a  family of shrinking round spheres.

By Lemma \ref{compact-R-decay},  the scalar curvature of  $(M, g)$ decays to zero uniformly.  Then $(M, g)$ has unique equilibrium point $o$ by the fact  ${\rm Ric}(g)>0$.   Thus  the level set $\Sigma_r=\{f(x)=r\}$  is a closed manifold for any $r>0$, and it is diffeomorphic to $S^{n-1}$ (cf. \cite[Lemma 2.1]{DZ-JEMS}).

 On the other hand,   as in the proof of \cite[Lemma 2.6]{ZZ-4d},     the level sets
    $(\Sigma_{f(q_i)},  \bar g_{q_i};  q_i)$
      converge   subsequently to  $(N', h'(0); q_\infty)$
      w.r.t.  the induced metric  $\bar g_{q_i}$  on $\Sigma_{f(q_i)}$  by  $g_{q_i}$.  Since each $\Sigma_{f(q_i)}$ is diffeomorphic to $S^{n-1}$,    $N'$ is also  diffeomorphic to $S^{n-1}$. Thus   $(N',h'(t))$ is a  family of shrinking round spheres.

     By the above argument, we see that
  the  condition (ii)  in \cite[ Definition 0.1]{ZZ-4d} is satisfied. Thus  by \cite[Lemma 6.5]{DZ-SCM},  $(M,g)$ is asymptotically cylindrical.  It follows that  $(M,g)$ is isometric to the Bryant Ricci soliton up to scaling by \cite{Bre-high}. Hence,  the theorem  is proved.

 \end{proof}


\begin{thebibliography}{99}
 %
 \bibitem{ABDS} Angenent, S., Brendle,  S., Daskalopoulos, P. and  Sesum, N., \textit{Unique asymptotics of compact ancient solutions to three-dimensional Ricci flow}, \emph{Comm. Pure Appl. Math.}, \textbf{75} (2022), 1032-1073.

  \bibitem{Bre-high} Brendle,  S., \textit{Rotational symmetry of Ricci solitons in higher dimensions}, \emph{J. Differential Geom.}, \textbf{97} (2014), 191-214.

  \bibitem{Bre-3d-noncpt} Brendle,  S., \textit{Ancient solutions to the Ricci flow in dimension 3}, \emph{Acta Math.}, \textbf{225} (2020), 1-102.

  \bibitem{BW} Brendle, S. and  Schoen, R., \textit{Manifolds with 1/4-pinched curvature are space forms}, \emph{J. Amer. Math. Soc.}, \textbf{22} (2009), 287-307.

  \bibitem{BK} Bamler, R. and Kleiner, B., \textit{On the rotational symmetry of 3-dimensional $\kappa $-solutions}, \emph{J. Reine Angew. Math.}, \textbf{779} (2021), 37-55.


  \bibitem{CDM} Chow,  B.,  Deng,  Y. and Ma,  Z.,  \textit{On four-dimensional steady gradient Ricci solitons that dimension reduce}, \emph{Adv. Math.}, \textbf{403} (2022), 61 pp.

  \bibitem{CLN} Chow,  B.,  Lu,  P.  and Ni,  L.,  \textit{Hamilton's Ricci flow}, American Mathematical Soc., 2006.

  \bibitem{DZ-JEMS} Deng,  Y. and  Zhu,  X. H.,  \textit{Higher dimensional steady Ricci solitons with linear curvature decay}, \emph{J. Eur. Math. Soc. (JEMS)}, \textbf{22}  (2020), 4097-4120.

  \bibitem{DZ-SCM} Deng,  Y.  and Zhu,  X. H.,  \textit{Classification of gradient steady Ricci solitons with linear curvature decay}, \emph{Sci. China Math.}, \textbf{63} (2020), 135-154.

   \bibitem{DZ-epsilon} Deng,  Y.  and Zhu,  X. H.,  \textit{Steady Ricci solitons with horizontally $\epsilon$-pinched Ricci curvature}, \emph{Sci. China Math.}, \textbf{64} (2021), 1411-1428.

   \bibitem{Yi-flyingwings} Lai,  Y., \textit{A family of 3d steady gradient solitons that are flying wings}, arXiv:2010.07272, 2020.

  \bibitem{Ni} Ni, L., \textit{Closed type I ancient solutions to Ricci flow}, \emph{Recent advances in geometric analysis}, ALM, vol.11 (2009), 147-150.

   \bibitem{P2} Perelman,  G.,  \textit{Ricci flow with surgery on Three-Manifolds}, arXiv:0303109, 2003.


  \bibitem{ZZ-4d} Zhao,  Z. Y. and Zhu, X. H., \textit{4d steady gradient Ricci solitons with nonnegative curvature away from a compact set}, arXiv:2310.12529, 2023.
   \end{thebibliography}
  \end{document}